\documentclass[11pt,reqno]{amsart}
\usepackage{amssymb}
\usepackage{amsmath}
\usepackage{enumitem}
\usepackage{mathrsfs}
\usepackage[all]{xy}
\usepackage{marginnote}
\usepackage{hyperref}

\usepackage[margin=1.25in]{geometry}

\RequirePackage{mathrsfs}
\usepackage[OT2,T1]{fontenc}
\DeclareSymbolFont{cyrletters}{OT2}{wncyr}{m}{n}
\DeclareMathSymbol{\Sha}{\mathalpha}{cyrletters}{"58}

\newtheorem{theorem}{Theorem}[section]
\newtheorem{lemma}[theorem]{Lemma}
\newtheorem{proposition}[theorem]{Proposition}
\newtheorem{corollary}[theorem]{Corollary}

\theoremstyle{definition}
\newtheorem*{ack}{Acknowledgements}
\newtheorem*{con}{Conventions}
\newtheorem{remark}[theorem]{Remark}

\newtheorem{definition}[theorem]{Definition}
\newtheorem{question}[theorem]{Question}

\numberwithin{equation}{section} \numberwithin{figure}{section}

\DeclareMathOperator{\Pic}{Pic} 
\DeclareMathOperator{\Gal}{Gal} 
\DeclareMathOperator{\Aut}{Aut} 

\DeclareMathOperator{\Spec}{Spec}

\DeclareMathOperator{\an}{an}
 \DeclareMathOperator{\Bl}{Bl}

\DeclareMathOperator{\rank}{rank}

\DeclareMathOperator{\Hom}{Hom} 
\DeclareMathOperator{\im}{Im}   
 
\DeclareMathOperator{\chr}{char}

\DeclareMathOperator{\Isom}{Isom}

\newcommand{\Index}{r}
\newcommand{\Degree}{d}

\newcommand{\PGL}{\textrm{PGL}}

\newcommand{\Qbar}{\overline{\QQ}}

\newcommand\PP{\mathbb{P}}

\newcommand\ZZ{\mathbb{Z}}
\newcommand\NN{\mathbb{N}}
\newcommand\QQ{\mathbb{Q}}

\newcommand\CC{\mathbb{C}}

\newcommand\OO{\mathcal{O}}

\newcommand\Slist{\mathscr{S}_{ep}}
\newcommand\Tlist{\mathscr{T}_{or}}

\title[Fano threefolds and sextic surfaces]{Good reduction of Fano threefolds and sextic surfaces}

\author{Ariyan Javanpeykar}
\address{Ariyan Javanpeykar \\
Institut f\"{u}r Mathematik\\
Johannes Gutenberg-Universit\"{a}t Mainz\\
Staudingerweg 9, 55099 Mainz\\
Germany.}
\email{peykar@uni-mainz.de}

\author{Daniel Loughran}
\address{Daniel Loughran \\
Leibniz Universit\"{a}t Hannover,
Institut f\"{u}r Algebra, Zahlentheorie
    und Diskrete Mathematik\\
Welfengarten 1\\
30167 Hannover\\
Germany.}
\email{loughran@math.uni-hannover.de}

  %Stacks and moduli problems


\begin{document}
 
\begin{abstract}
	We investigate versions of the Shafarevich conjecture, 
	as proved for curves and abelian varieties by Faltings, for other classes of varieties.
	We first obtain analogues for certain Fano threefolds.
	We use these results to prove  the Shafarevich conjecture
	for smooth sextic surfaces, which appears to be the first non-trivial result in the literature
	on the arithmetic of such surfaces.
	Moreover, we exhibit certain moduli stacks of Fano varieties which are not hyperbolic,
	which allows us to show that the analogue of the Shafarevich conjecture does not always
	hold for Fano varieties.
	Our results also provide new examples for which the conjectures of Campana and Lang--Vojta hold.
\end{abstract}

\maketitle
\tableofcontents

\thispagestyle{empty}

\section{Introduction}
A   theorem of Faltings \cite{Faltings2} (formerly the Shafarevich conjecture), states that given a number field $K$, a finite set $S$ of finite places of $K$ and an integer $g$, there are only finitely many isomorphism classes of $g$-dimensional abelian varieties over $K$ with good reduction outside $S$. Analogues of this result have been obtained for several other classes of varieties, such as cubic fourfolds and very polarised hyperk\"{a}hler varieties of bounded degree \cite{Andre},  K$3$ surfaces \cite{She14}, del Pezzo surfaces \cite{Scholl}, flag varieties \cite{JL2}, certain surfaces of general type \cite{Jav15}, and complete intersections of Hodge level at most $1$ \cite{JL}.

In this paper we study analogues of the Shafarevich conjecture for Fano threefolds. 
To state our results, we recall some important geometric invariants. Let $X$ be a Fano threefold over a field $k$. We define the (geometric) \emph{Picard number} of $X$ to be the rank of the Picard group of $X_{\bar{k}}$. We define its (geometric) \emph{index} $\Index(X)$ to be the largest $r \in \NN$ such that $-K_X$ is divisible by $r$ in $\Pic X_{\bar{k}}$, and its (geometric) \emph{degree} to be $(-K_X)^3/\Index(X)^3$. These invariants naturally arise in the classification of Fano threefolds (see \cite{IskPro} for details). We recall in Section \ref{section:clas} the cases of the classification relevant to us in this paper.
Our first result is the following.

\begin{theorem}\label{theorem1}
	Let $K$ be a number field and  let $S$ be a finite set of  finite places of $K$.  Then the set of $K$-isomorphism classes of   Fano threefolds of  Picard number $1$ and index at least $2$ over $K$
	 with good reduction outside of $S$  is finite.
\end{theorem}

Here by good reduction, we mean \emph{good reduction as a Fano variety} (see Definition \ref{def:good_reduction}).
Certain cases of Theorem \ref{theorem1} are already known, and follow from our work \cite{JL} on complete intersections
(e.g.~cubic threefolds and intersections of two quadrics in $\PP^5$). New cases include hypersurfaces of degree
$4$ and $6$ in the weighted projective spaces $\PP(2,1,1,1,1)$ and $\PP(3,2,1,1,1)$, respectively.

We also address some Fano threefolds of Picard number $1$ and index $1$. Recall that for such a threefold $X$, its \emph{genus} is defined to be the integer $(-K_X)^3/2 + 1$. Our work on complete intersections  \cite{JL}  proves versions of the Shafarevich conjecture for such Fano threefolds of genus $4$ and $5$. Our next result handles the case of genus $2$.
 
 \begin{theorem}\label{theorem2}
 Let $K$ be a number field and  let $S$ be a finite set of  finite places of $K$.  Then the set of $K$-isomorphism classes of   Fano threefolds of Picard number $1$, index $1$ and genus $2$  over $K$
	 with good reduction outside of $S$ is finite.
 \end{theorem}
  
Whilst these  Fano threefolds are interesting in their own right, their primary interest for us  stems from their relationship to sextic surfaces. Namely, such Fano threefolds are double covers of $\PP^3$ ramified along a smooth sextic surface, and moreover every smooth sextic surface arises this way (see Remark \ref{remark:phil}). Over an algebraically closed field any such sextic is uniquely determined by the associated threefold, however this may not be the case over non-algebraically closed fields. Nevertheless, exploiting this relationship we are able to use Theorem \ref{theorem2} to establish the following.
 
  \begin{theorem}\label{theorem3}
 Let $K$ be a number field and  let $S$ be a finite set of  finite places of $K$.  Then the set of $K$-isomorphism classes of   smooth sextic surfaces in $\mathbb P^2_K$ with good reduction outside of $S$  is finite.
 \end{theorem}
 
 Here by good reduction, we mean \emph{good reduction as a sextic surface} (see Section \ref{sec:sextic} for precise definitions). This result proves a special case of the Shafarevich conjecture for smooth complete intersections  posed by the authors in \cite[Conj.~1.4]{JL}.
 
 To the authors' knowledge, Theorem \ref{theorem3} is the first non-trivial result in the literature  on the arithmetic of smooth sextic surfaces. Note that these surfaces do not have K3-type (having middle Hodge numbers $10,86,$ and $10$), hence the usual Kuga-Satake construction \cite{DeligneK3, Moo15} cannot be applied here, as used to prove the Shafarevich conjecture for K3 surfaces \cite{Andre, She14}, say. %Moreover, the methods of Deligne in \cite[\S 7]{DeligneK3} show that the \emph{motive} of a general sextic in $\PP^3_\CC$ is not \emph{abelian} \cite[Defn.~7.2]{DeligneK3}.
 
There is now a growing body of evidence which suggests that the analogue of the Shafarevich conjecture could always hold for Fano varieties. It turns out however that this is not the case, as our next result shows.

\begin{theorem}\label{thm:CE}
There exists a finite set of finite places $S$ of $\QQ$ such that the set of $\bar{\QQ}$-isomorphism classes of Fano threefolds of Picard number $2$ over $\QQ$ with good reduction outside $S$ is infinite.
\end{theorem}

We prove Theorem \ref{thm:CE} by constructing an explicit family of counter-examples (see Theorem \ref{theorem:counter-example}). We briefly
explain this construction now, as it is very simple. Let $X$ be a smooth intersection of two quadrics
in $\PP^5_{\QQ}$ which contains infinitely many lines. 
The blow-up of $X$ at a line is a Fano threefold of Picard number $2$. However, the lines
on $X$ are parametrised by an abelian surface, namely the intermediate Jacobian of $X$. Under our assumptions
the intermediate Jacobian has infinitely many integral points, which implies that we obtain infinitely many  Fano threefolds
with good reduction outside some finite set of primes. We then show that we obtain infinitely many $\bar{\QQ}$-isomorphism classes using the finiteness of $\Aut(X)$.

Note that this method breaks down for del Pezzo surfaces, where a version of the Shafarevich conjecture was shown to hold by Scholl \cite{Scholl}. Here to obtain non-trivial moduli one must blow-up at least $5$ points of $\PP^2$
in \emph{general position}. However, the configuration space of points in general position turns out to be
hyperbolic and to have only finitely many integral points. This phenomenon does not occur for intersections of two quadrics as no general position hypotheses are required:
one is allowed to blow-up \emph{any} line.

Our finiteness results (Theorems \ref{theorem1}, \ref{theorem2} and \ref{theorem3}) may be interpreted as special cases of the \emph{Lang-Vojta conjecture} \cite{Abr, CHM}.
 To explain this relationship in the setting of Faltings's work \cite{Faltings2} (cf.~\cite[Ex.~F.5.3.7]{HindrySilverman}), let $\mathcal A_g$ be the stack of $g$-dimensional principally polarised abelian schemes over $\ZZ$. Since the Siegel upper half-space $\mathbb{H}_g$ is a bounded domain and $\mathcal {A}_{g,\CC}^{\textrm{an}} = [\mathrm{Sp}_{2g}(\ZZ)\backslash \mathbb H_g]$, 
 all holomorphic maps $\mathbb C\to \mathcal A_{g,\CC}^{\textrm{an}}$ are constant, i.e.~$\mathcal A_{g,\CC}$ is Brody hyperbolic. In particular, in  the light of the Lang-Vojta conjecture, one expects that $\mathcal A_{g}$ has only finitely many (isomorphism classes of) integral points, as confirmed by the results of Faltings.
Many naturally occuring moduli stacks are hyperbolic (e.g.~the stack of canonically polarised varieties \cite{VZ}), and ignoring stacky issues, the Lang-Vojta conjecture implies that such hyperbolic moduli stacks should have only finitely many  (isomorphism classes of) integral points. Our finiteness results fit into this general framework (see Remark \ref{remark:lang_vojta}).

Theorem \ref{thm:CE} shows that some components of the stack of Fano varieties have infinitely many integral points. The threefolds we use to prove Theorem \ref{thm:CE} appear as No.~19 of Table 12.3 of \cite{IskPro} in the classification of Fano threefolds, and have index $1$, degree $26$ and third Betti number $4$. 

However our final result (Theorem \ref{thm:non_density1}) shows that, nevertheless, the Diophantine properties of this moduli stack are in accordance with Campana's generalised version of the Lang-Vojta conjecture \cite[Conj.~13.23]{Campana}.  We will content ourselves here with an imprecise version of our result, and refer the reader to Theorem \ref{thm:non_density2} for a more precise statement, and to Remark \ref{remark:campana} for an explanation of the relationship to Campana's conjecture.

\begin{theorem}\label{thm:non_density1}
Let $K$ be a number field and let $S$ be a finite set of finite places of $K$. Then the set of $\OO_K[S^{-1}]$-points in the moduli stack of Fano threefolds of Picard number $2$, index $1$,  degree $26$ and  third Betti number $4$ is not Zariski dense.
\end{theorem}

We now give an outline of the ingredients of the proofs of Theorems \ref{theorem1} and \ref{theorem2}.
For certain $r$ and $d$, it turns out that all Fano threefolds with Picard number $1$, index $r$ and degree $d$ satisfy the infinitesimal Torelli property. Moreover, for certain $r$ and $d$, it is known that such Fano threefolds are non-rational. In these cases, a classical result of Matsusaka and Mumford \cite{MatMum} can be used to show that the stack parametrizing such Fano threefolds is a smooth finite type separated Deligne-Mumford stack over $\CC$. By separatedness, the main difficulty becomes showing that the set of ``periods'' (i.e.~integral polarised Hodge structures) associated to all Fano threefolds over $K$ with good reduction outside $S$ is finite. To do so, we use the theory of the \emph{intermediate Jacobian}. This is usually constructed via transcendental methods as a complex torus, however, for Fano threefolds, a recent result of Achter--Casalaina-Martin--Vial \cite{ACM} shows that it descends to subfields of $\CC$. This theory, once properly performed and combined with Faltings's finiteness results for abelian varieties \cite{Faltings2}, yields the proof of Theorems \ref{theorem1} and Theorem \ref{theorem2} in most cases. This method is based on our proof of the Shafarevich conjecture for smooth complete intersections of Hodge level $1$ \cite{JL}. The main remaining difficult case is that of index $2$ and degree $5$. Here the stack of Fano threefolds is neither separated nor Deligne-Mumford. Nevertheless, such Fano threefolds are infinitesimally rigid (Corollary \ref{cor:vanishing_of_H1}). This property allows us to prove the result by appealing to finiteness results on cohomology sets due to Gille--Moret-Bailly \cite{GilleMoretBailly}. This is similar to the case of flag varieties treated in \cite{JL2}. 

With current tools, these methods do not yield new results for other Fano threefolds of Picard number $1$ and index $1$. The issue is that either one does not know an infinitesimal Torelli theorem, or the moduli is non-separated but non-discrete. For example, 
Fano threefolds of Picard number $1$, index $1$ and genus $12$ can have positive-dimensional automorphism groups \cite[Thm.~5.2.13]{IskPro}, so their moduli is non-separated. Moreover, their moduli is $6$-dimensional \cite[Thm.~5.2.11]{IskPro} and 
they have trivial intermediate Jacobian, so the infinitesimal Torelli theorem clearly does not hold here. We therefore finish with the following question.
\begin{question}\label{question}
	 Let $K$ be a number field and  let $S$ be a finite set of  finite places of $K$. Is the set of $K$-isomorphism classes of Fano threefolds over $K$ of Picard number $1$ with good reduction outside of $S$ finite?
\end{question}

\subsection*{Outline of the paper} In Section \ref{section:fanos} we study the general properties of Fano varieties in families, in particular their invariants and moduli. In Section \ref{section:Fanothreefolds} we specialise to the case of Fano threefolds of Picard number $1$, where we study finer properties such as separatedness of the moduli stacks and Torelli-type problems. In Section \ref{section:finiteness} we use these geometric results to deduce our main theorems (Theorems \ref{theorem1}, \ref{theorem2} and \ref{theorem3}). Finally in Section \ref{section:CE} we construct the counter-examples necessary for Theorem \ref{thm:CE}, and prove Theorem \ref{thm:non_density1}.

\begin{ack} We thank Olivier Benoist for answering our many questions on the separatedness of the stack of Fano varieties.  We are grateful to Brian Conrad, Jack Hall, and David Rydh for  comments and corrections on the structure of the stack of Fano varieties.  We thank Jean-Beno\^it Bost, Francois Charles,  Bas Edixhoven,   Frank Gounelas, David Holmes, Robin de Jong, Remke Kloosterman,  Sasha Kuznetzov, Manfred Lehn, Stefan M\"uller-Stach, Cl{\'e}lia Pech, Jason Starr,  Duco van Straten, Ronan Terpereau, Chenyang Xu and  Kang Zuo for interesting and useful discussions. We also thank the referee for some useful comments. The first named author gratefully acknowledges support of SFB/Transregio 45.
\end{ack}

\begin{con}

For a field $k$, we let $k\to \bar k$ be an algebraic closure.

A \textit{Dedekind scheme} is an integral noetherian normal one-dimensional scheme. 

For a number field $K$, we denote by $\OO_K$ its ring of integers.
If $S$ is a finite set of finite places of $K$, we let $\OO_K[S^{-1}]$ denote the localization of $\OO_K$ at $S$.

An \textit{arithmetic scheme} is an integral regular finite type  flat scheme over $\ZZ$.  Note that if $B$ is a one-dimensional arithmetic scheme, then there exist a number field $K$ and a finite set of finite places $S$ of $K$ such that $B \cong \Spec \OO_{K}[S^{-1}]$. 

If $B$ is a  scheme and $N\neq 0$ an integer, we write $B_{\ZZ[1/N]}$ for $B \times_{\Spec \ZZ} \Spec \ZZ[1/N]$.
 For $b\in B$, we let $\kappa(b)$ denote the residue field of $b$. We say that an integer $N$ is invertible on $B$ if, for all $b$ in $B$, the integer $N$ is invertible in $\kappa(b)$. 
 
A variety over a field $k$ is a   finite type  $k$-scheme. If $X$ is a variety over $k$, we let $\Theta_{X/k}$ (or simply $\Theta_X$) denote the tangent sheaf of $X$ over $k$.

For a noetherian scheme $X$,
we denote by $X^{(1)}$ its set of codimension $1$ points.

For a tuple $(a_0, \ldots, a_n)$ of integers and a scheme $B$, we denote by $\PP_B(a_0,\ldots,a_n)$ the corresponding weighted projective space over $B$.

\end{con}

\section{Families of Fano varieties}\label{section:fanos}   In this section we study the geometry of Fano varieties and their moduli.

\begin{definition}\label{def:fano_scheme}
Let $k$ be a field. A \emph{Fano variety over $k$} is a smooth proper geometrically integral variety over $k$ with ample anticanonical bundle.

Let $B$ be a scheme. A \emph{Fano scheme over $B$} is a smooth proper morphism $X\to B$ of schemes whose fibres are Fano varieties.  A Fano $n$-fold over $B$ is a Fano scheme of relative dimension $n$.
\end{definition}

\subsection{The stack of Fano varieties}\label{section:stacks}
Concerning   algebraic stacks, we will use the conventions of the \emph{Stacks project}; see \cite[Tag 026N]{stacks-project}.

We  first discuss the moduli of polarised varieties. As is written in \cite[\S3]{dJHS}, the correct approach to use when doing moduli of algebraic varieties is to use families of varieties where the total space is an algebraic space.  However, when studying families of polarised varieties, the use of algebraic spaces can be avoided. Indeed, if $B$ is a scheme and  $X\to B$ is a smooth proper morphism of algebraic spaces which admits a relatively ample line bundle, then $X$ is a scheme.

Let $\textsc{Pol}$ be the fibred category (in groupoids over $(\mathrm{Sch})_{{fppf}}$)  with objects triples $(U,X\to U, L)$, where $U$ is a scheme, $X\to U$ is a flat proper locally finitely presented morphism of schemes and $L$ is a relatively ample line bundle  on $X$. A morphism from $(U,X\to U,L)$ to $(V,Y\to V, M)$ in $\textsc{Pol}$ is given by a triple  $(f,g,h)$, where $f:U\to V$ is a morphism of schemes, $g:X\to Y\times_V U$ is an isomorphism of $U$-schemes and $h$ is an isomorphism from the invertible sheaf $L$ on $X$ to the pull-back  of $M$ (cf.~\cite[\S 4]{Starr06}). 

 \begin{lemma}\label{lem:pol}
 The fibred category $\textsc{Pol}$ is an algebraic stack, locally of finite type over $\ZZ$, whose diagonal is affine and of finite type.
 \end{lemma}
 
 \begin{proof}
As is mentioned in \cite[Rem.~4.3]{Starr06}, one can use the theory of Hilbert schemes to give a smooth presentation of $\textsc{Pol}$.  Alternatively, one can verify Artin's axioms for a stack to be limit preserving and algebraic (see \cite{Artin}, \cite[Tag 07XJ]{stacks-project} and \cite[Tag 07Y3]{stacks-project}). This is well explained in the proof of  \cite[Prop.~4.2]{Starr06} (see also \cite[Thm.~8.1]{Hall1}).   Finally, the diagonal is shown to be affine and of finite type in 
\cite[\S 2.1]{dJS10}. 
\end{proof}

Let $\textsc{Fano}$ be the fibred category  with  objects Fano schemes  $f:X\to B$ (Definition \ref{def:fano_scheme}), and morphisms given by Cartesian diagrams. The fibred category $\textsc{Fano}$ is a stack (in groupoids over $(\mathrm{Sch})_{{fppf}}$); it is the \emph{stack of Fano varieties}.

There is a  natural morphism of stacks $\textsc{Fano}\to \textsc{Pol}$ which associates to a Fano scheme $X\to S$ the triple $\smash{(S,X\to S, \omega_{X/S}^{-1})}$. We emphasise that this morphism is not fully faithful. Indeed, if $X\to S$ is a Fano scheme, then  the polarised variety $\smash{(S,X\to S, \omega_{X/S}^{-1})}$ has more automorphisms than the Fano scheme $X\to S$ (given  by the identity on $X\to S$ and scalar multiplication of the line bundle $\omega_{X/S}^{-1}$). 

We will say that a morphism of algebraic stacks $X\to Y$ is quasi-affine if it is representable by schemes and, for all schemes $S\to Y$, the morphism $X\times_Y S\to S$ is a quasi-affine morphism of schemes.

\begin{lemma}\label{lem:fano_to_pol}
The natural morphism $\textsc{Fano}\to \textsc{Pol}$ is  quasi-affine.
\end{lemma}

\begin{proof} 
	Let $\textsc{Pol-Fano}$ be the  substack of $\textsc{Pol}$ whose objects are triples $(S, f:X\to S,   L)$ in $\textsc{Pol}$ such that $f$ is smooth and $\omega_{X/S}^{-1}$ is relatively ample. As smoothness is an open condition and relative ampleness of a line bundle is an open condition, it follows  that $\textsc{Pol-Fano}$ is an open substack of $\textsc{Pol}$. 
	
	Note that  $\textsc{Fano}\to \textsc{Pol}$ factors via a morphism $\textsc{Fano}\to \textsc{Pol-Fano}$. Therefore, as the composition of an affine morphism with an open immersion is quasi-affine, to prove the lemma, it suffices to show that $\textsc{Fano}\to \textsc{Pol-Fano}$ is an affine morphism.
	
%	If $S$ is a scheme,  $X\to S$ is a flat proper locally finitely presented morphism of schemes, and $L_1, L_2$ are relatively ample line bundles on $X$, then we let $\mathrm{Hom}_{\mathcal O_X/S}(L_1,L_2)$ be as in \cite[Thm.~D]{Hall2}.  Moreover, we let $\mathrm{Isom}_{\mathcal O_X/S}(L_1,L_2)$ be the subfunctor of $\mathrm{Hom}_{\mathcal O_X/S}(L_1,L_2)$ which parametrises isomorphisms from $L_1$ to $L_2$.
	
	Let $S$ be a scheme and let $S\to \textsc{Pol-Fano}$ be a morphism. Let $(S, f:X\to S, L')$ be the corresponding object of $\textsc{Pol-Fano}$. Write $L:= \omega_{X/S}^{-1}$.  For all $S$-schemes $T$, the $T$-points of  the scheme $\textsc{Fano}\times_{\textsc{Fano-Pol}} S$ are canonically the $T$-points of the $S$-scheme $\mathrm{Isom}_{\mathcal O_X/S}(L,L')$. Here we let $\mathrm{Hom}_{\mathcal O_X/S}(L,L')$ be as in \cite[Thm.~D]{Hall2}, and denote by $\mathrm{Isom}_{\mathcal O_X/S}(L,L')$ the subfunctor of  $\mathrm{Hom}_{\mathcal O_X/S}(L,L')$   parametrising isomorphisms from $L$ to $L'$. To conclude    the proof,  it suffices to show that $\mathrm{Isom}_{\mathcal O_X/S}(L,L')$ is affine. 
	
%	Let $S\to \mathrm{Hom}_{\mathcal O_X/S}(L,L) \times \mathrm{Hom}_{\mathcal O_X/S}(L',L')$ be the identity section. 
To do so, we note that the following diagram 
\begin{equation*} 
\begin{split}%\label{eqn:Hall}
\xymatrix{ \mathrm{Isom}_{\mathcal O_X/S}(L,L')\ar[r] \ar[d]  & S \ar[d] \\ \ar[r] \mathrm{Hom}_{\mathcal O_X/S}(L,L') \times_S \mathrm{Hom}_{\mathcal O_X/S}(L',L) & \mathrm{Hom}_{\mathcal O_X/S}(L,L)\times_S \mathrm{Hom}_{\mathcal O_X/S}(L',L')}
\end{split}
\end{equation*} is Cartesian. Here the bottom horizontal arrow is given by ``composition'', the left vertical arrow sends an isomorphism $\phi$ to $(\phi, \phi^{-1})$, and the right vertical arrow is the identity section.
Since the identity section 
is a closed immersion, we conclude that $$\mathrm{Isom}_{\mathcal O_X/S}(L,L') \to \mathrm{Hom}_{\mathcal O_X/S}(L,L') \times_S \mathrm{Hom}_{\mathcal O_X/S}(L',L) $$ is a closed immersion. As      $\mathrm{Hom}_{\mathcal O_X/S}(L,L^\prime)$ and $\mathrm{Hom}_{\mathcal O_X/S}(L',L)$ are affine ~\cite[Thm.~D]{Hall2}, it follows  that  $ \mathrm{Isom}_{\mathcal O_X/S}(L,L')$ is affine, as required.
	%What follows is the version 	of the proof in the submitted version to Pisa.
	%As relative ampleness is  open  on the base, it suffices to show that if $f:X\to B$ is a Fano scheme and $L, L^\prime$ are relatively ample invertible sheaves on $X$, then $\Isom_{\mathcal O_X/B}(L,L^\prime)$ is a quasi-affine scheme. However, as $\Isom_{\mathcal O_X/B}(L,L^\prime)$ is an open subscheme of the scheme $\mathrm{Hom}_{\mathcal O_X/B}(L,L^\prime)$, it suffices to note that $\mathrm{Hom}_{\mathcal O_X/B}(L,L^\prime)$ is affine~\cite[Thm.~D]{Hall2}.
\end{proof}  

\begin{lemma}\label{lem:algebraic}
The stack  $\textsc{Fano}$ is an algebraic stack, locally of finite type over $\ZZ$, whose diagonal is  affine and of finite type.
\end{lemma}
\begin{proof}
By Lemma \ref{lem:fano_to_pol}, the  natural morphism of stacks $\textsc{Fano}\to \textsc{Pol}$  is quasi-affine. As the diagonal of a quasi-affine morphism of schemes is a closed immersion, the result follows from Lemma \ref{lem:pol}.
 \end{proof}
 
\begin{lemma}\label{lem:smoothness}
The algebraic stack $\textsc{Fano}_\QQ $ is smooth over $\QQ$.
\end{lemma}
\begin{proof} 
Fano varieties are unobstructed in characteristic zero; see \cite{Ran} or \cite{Sano14}. In particular, the stack $\textsc{Fano}_{\QQ}$ is formally smooth over $\QQ$. As $\textsc{Fano}_\QQ$ is  locally of  finite type over $\QQ$ (Lemma \ref{lem:algebraic}), the lemma follows.
\end{proof}
 
 \begin{corollary}\label{cor:autom_groups_are_affine} Let $B$ be a scheme. Let  $X\to B$ and $Y\to B$ be Fano schemes. Then the scheme $\Isom_B(X,Y)$ is affine and of finite type over $B$.
\end{corollary}
 \begin{proof}
The diagonal of $\textsc{Fano}$ is affine and of finite type (Lemma \ref{lem:algebraic}).
\end{proof}

\subsection{Invariants of Fano varieties in families}

The aim of this section is to show that the invariants of Fano varieties are constant in a family, under suitable assumptions.

Recall that, for $X$  a Fano variety over a field $k$, its Picard group  $\Pic(X)$   is finitely generated \cite[Thm.~1.1]{SB97}.

\begin{definition}[Invariants]  Let $X$ be a Fano $n$-fold over an  algebraically closed field~$k$.  
\begin{itemize}
	\item The \emph{Picard number} of $X$ is defined to be   $\rho(X) = \rank_{\ZZ} \Pic X$.
	\item The \emph{index} of $X$  is defined to be $$\Index(X) = \max\{m \in \NN: -K_X/m \in \Pic X\}.$$
	
	\item A \emph{fundamental divisor} for $X$ is a divisor $H_X$ on $X$ for which $\Index(X)H_X = -K_X$.
	\item The \emph{degree} of $X$ is defined to be  $\Degree(X) = (-K_X)^n/\Index(X)^n\,=H_X^n.$	
\end{itemize}
If $X$ is a Fano variety over an  {arbitrary field}  $k$, we define $\rho(X)$ (resp.~$r(X)$, $d(X)$) to be $\rho(X_{\bar k})$ (resp.~$r(X_{\bar k})$, $d(X_{\bar k})$).
 \end{definition}

\begin{definition}\label{defn:pic} A Fano scheme $X\to B$  has Picard number $\rho$ (resp.~index $r$, degree $d$)  if all its geometric fibres have Picard number $\rho$ (resp.~index $r$, degree $d$).
\end{definition}

A Fano variety over a field $k$ is called \emph{split} if the natural map $\Pic X \to \Pic X_{\bar{k}}$ is an isomorphism.

\begin{lemma}\label{lem:picard_rank} 
Let $n\in \NN$ and let $B$ be the spectrum of a regular local Noetherian ring with generic point $\eta$ and closed point $b$. Assume that $\chr \kappa(\eta) = 0 $ and that either $\chr \kappa(b) = 0$ or $\chr \kappa(b) >n$. Let $f: X \to B$ be a Fano $n$-fold such that $X_\eta$ and $X_b$ are both split. Then the natural maps 
\begin{equation} \label{def:specialise}
	\Pic X \to \Pic X_\eta, \quad 
	\Pic X \to \Pic X_b
\end{equation} 
are isomorphisms.
\end{lemma}
\begin{proof}
Note that, as $X$ is smooth over a regular scheme, the scheme $X$ is regular.
Hence the fact that the first map is an isomorphism follows from \cite[Lem.~3.1.1]{Har94}. For the second map, we may pass to a completion and 
	assume that $B$ is complete. Indeed, as $X_\eta$ and $X_b$ are both split,
	the Picard groups are left unchanged. By our assumptions on the residue characteristics of $B$, we can use Kodaira vanishing (see \cite{DI} or \cite[Cor.~5.2]{EVbook}) to find that $\mathrm{H}^1(X_b, \OO_{X_b})=\mathrm{H}^2(X_b, \OO_{X_b})=0$.
	Thus the Picard scheme is formally \'{e}tale \cite[Cor.~6.1]{GrothGFGA},
	which proves the result.
\end{proof}

\begin{proposition} \label{prop:invariants}
	Let $n \in \NN$ and let $B$ be a connected scheme with all generic points of characteristic $0$
	and all other points either of characteristic $0$ or characteristic greater than $n$.
	Let $f: X \to B$ be a Fano $n$-fold. Then the Picard number, the index, and the degree are constant on the   fibres of $f$.
\end{proposition}
\begin{proof}
	To prove the result, we may and do assume that $B$ is the spectrum of a regular local ring  with generic point $\eta$ and closed point $b$. We may also assume that  $X_\eta$ and $X_b$ are both split. 
	 We need only show that the invariants over the generic point $\eta$ and the closed point $b$
	agree.  
	
	First consider the relative anticanonical bundle $\omega_{X/B}^{-1}$. This induces the anticanonical
	bundle on each fibre. This observation, together with Lemma \ref{lem:picard_rank}, 
	implies that the Picard number and the index are constant. 
	For the degree, it suffices to note that $-K_{X_c}^n$ is constant as $f$ is flat and proper.	
\end{proof}

\begin{remark}
Note that in the statement of Proposition \ref{prop:invariants} it is crucial that we defined the invariants geometrically. Consider the quadric
surface
$$x_0^2 + x_1^2 = x_2x_3 \quad \subset \mathbb P^3_\ZZ.$$
Here the generic fibre has Picard group $\ZZ$, whereas the fibre over any prime which is $1 \bmod 4$ has Picard group $\ZZ^2$. Similarly, consider the conic
$$x_0^2 + x_1^2 + x_2^2 = 0 \quad \subset \mathbb P^2_\ZZ.$$
The anticanonical divisor on the generic fibre is not divisible, but modulo all odd primes the anticanonical
divisor becomes divisible by $2$.
\end{remark}

\begin{remark}
	It is quite possible that versions of Proposition \ref{prop:invariants} hold in greater generality
	(i.e.~without the restrictive hypothesis on the characteristics), and can be deduced using
	the ``Bloch-Srinivas method'' (see \cite{BS83} or \cite[\S 10.2]{VoisinII}), and the rational chain connectedness of Fano varieties \cite{CampanaRCC, KMM}. We do not explore these further in this paper, as our more elementary
	result %(Proposition \ref{prop:invariants})  
	will be sufficient for our purposes. 
\end{remark}

\section{Fano threefolds} \label{section:Fanothreefolds}
\subsection{The classification of   Fano threefolds}\label{section:clas}

We will frequently use the classification of Fano threefolds $X$ with Picard number $1$ over an algebraically closed field of arbitrary characteristic. There are precisely $17$ families; see \cite[Table 12.2]{IskPro}, \cite{Megyesi} and \cite{SB97}. These are classified according to their index and degree (Definition \ref{defn:pic}). In index $1$ however, it is customary to work instead with the genus $g(X) = \Degree(X)/2 + 1$. We summarise now the cases which will be relevant in this paper.

Let $X$ be a Fano threefold over an algebraically closed field $k$ with $\rho(X) = 1$. 
\begin{itemize}
	\item If $\Index(X) =4$, then $X \cong \mathbb P^3_k$.
	\item If $\Index(X) =3$, then $X$ is isomorphic to a smooth quadric in $\mathbb P^4_k$.
	\item If $\Index(X) =2$, then $1\leq \Degree(X) \leq 5$. We have the following possibilities.
	\begin{itemize}
		\item[] $\Degree(X) = 1:$ Hypersurface of degree $6$ in $\PP_k(3,2,1,1,1)$.
		\item[] $\Degree(X) = 2:$ Hypersurface of degree $4$ in $\PP_k(2,1,1,1,1)$.
		\item[] $\Degree(X) = 3:$ Cubic threefold.
		\item[] $\Degree(X) = 4:$ Complete intersection of two quadrics in $\PP^5_k$.				
		\item[] $\Degree(X) = 5:$ A section of the Grassmannian 
		$\mathrm{Gr}(2,5)\subset \mathbb P^9_k$ by a linear subspace of codimension $3$.	
	\end{itemize}
	\item If $\Index(X) =1$, then $1\leq g(X) \leq 12$ and $g(X) \neq 11$. We will be interested in the following special cases.
	\begin{itemize}
		\item[] $g(X) = 2:$ Hypersurface of degree $6$ in $\PP_k(3,1,1,1,1)$.
		\item[] $g(X) = 3:$ Quartic threefold, or a double cover of a smooth quadric 
		$Q\subset \mathbb P^4_k$ ramified along a divisor of degree $8$ in $Q$.
		\item[] $g(X) = 4:$ Complete intersection of a quadric and a cubic in $\PP^5_k$.
		\item[] $g(X) = 5:$ Complete intersection of three quadrics in $\PP^6_k$.
	\end{itemize}
\end{itemize}

\begin{remark}
  Here it is crucial that we are working over an algebraically closed field. A non-split Fano threefold $X$ of Picard number $1$ over an arbitrary field $k$ may not be embeddable into a projective space in the above form, as the fundamental divisor of $X_{\bar{k}}$ may not be defined over $k$ when $\Index(X) \geq 2$ (consider a non-split Brauer-Severi threefold).
\end{remark}

\subsection{The stack of Fano threefolds with Picard number one}\label{section:fid}
  We will use the following result of Shepherd-Barron \cite{SB97} (that itself is an application of Ekedahl's work \cite{Ekedahl}) to prove a stronger version of Proposition \ref{prop:invariants} for Fano threefolds.

\begin{lemma}\label{lem:vanishing}
Let $X$ be a Fano threefold over a field $k$. Then $\mathrm{H}^1(X,\mathcal O_X) = \mathrm{H}^2(X,\mathcal O_X)=0$.\end{lemma}
\begin{proof}
This follows from \cite[Cor.~1.5]{SB97} (cf.~\cite[Cor.~2]{Megyesi}).
\end{proof}

\begin{lemma}\label{lem:openness}
Let $X\to B$ be a  Fano threefold over a connected scheme $B$. Then the Picard number, the index, and the degree are constant in the fibres of $X\to B$. 
\end{lemma}
\begin{proof} Use Lemma \ref{lem:vanishing} and the arguments in the proofs of   Lemma \ref{lem:picard_rank} and Proposition \ref{prop:invariants}. %More precisely, standard approximation, localization and slicing arguments show that we may assume $B$ to be integral regular and noetherian.
\end{proof}

Let $\mathcal F$ be the substack of $\textsc{Fano}$ whose objects are Fano threefolds with Picard number $1$ (Definition \ref{defn:pic}).

\begin{lemma}\label{lem:stack_F}
The stack $\mathcal F $ is an algebraic stack of finite type over $\ZZ$ with affine and finite type diagonal. The natural morphism $\mathcal F\to \textsc{Fano}$ is a representable open and  closed immersion, and the stack $\mathcal F_\QQ$ is smooth.
\end{lemma}
\begin{proof}  By Lemma \ref{lem:openness}, the forgetful morphism $\mathcal F\to \textsc{Fano}$ is a representable open and closed immersion.
 Therefore, it follows from Lemma \ref{lem:algebraic} that $\mathcal F$ is an algebraic stack, locally of finite type over $\ZZ$ with affine and finite type diagonal.

The stack $\mathcal{F}$ is of finite type by the classification of   Fano threefolds with geometric Picard number one (Section \ref{section:clas}) and finiteness properties of Hilbert schemes. 

   Finally, the smoothness of $\mathcal F_{\QQ}$ follows from the smoothness of $\textsc{Fano}_\QQ$ (Lemma \ref{lem:smoothness}) and the fact that $\mathcal F_{\QQ}\to \textsc{Fano}_{\QQ}$ is  an open immersion. This concludes the proof of the lemma.
\end{proof}

We now introduce certain substacks of $\mathcal{F}$. Let $r,d \in \NN$ and let $\mathcal F_{r,d}$ be the   substack of $\mathcal F$ with objects Fano threefolds $f:X\to B$ whose fibres have index $r$ and degree $d$. The stack $\mathcal F_{r,d}$ is non-empty if and only if $r$ and $d$ take the values arising in the classification (Section \ref{section:clas}). 

\begin{lemma}\label{lem:stack_Fid} The stack $\mathcal F_{r,d}$ is an algebraic stack of finite type over $\ZZ$ with affine and finite type diagonal, and the natural representable morphism $\mathcal F_{r,d} \to \mathcal F$ of stacks is an open and closed immersion. 
The substacks $\mathcal F_{r,d,\QQ}$ of the stack $\mathcal F_\QQ$ are smooth over $\QQ$. 
\end{lemma}
\begin{proof}
We find from Lemma \ref{lem:openness} that the forgetful morphism $\mathcal F_{r,d}\to \mathcal F$ is representable by an open and closed immersion, so that the lemma follows from Lemma \ref{lem:stack_F}.
\end{proof}

\begin{remark}
The properties of the stacks $\mathcal F_{r,d}$ vary  with $r$ and $d$. For example, the stack $\mathcal F_{2,3}$ is isomorphic to the stack $\mathcal C_{(3;3)}$ of smooth cubic threefolds, and $\mathcal F_{2,4}$ is isomorphic to the stack of three-dimensional smooth intersections of two quadrics $\mathcal C_{(2,2;3)}$ (see \cite{Ben13} for precise definitions of these stacks). By \cite[Thm.~1.6]{Ben13} the algebraic stack $\mathcal F_{2,3}$ is Deligne-Mumford over $\ZZ$, whereas $\mathcal F_{2,4}$ is only Deligne-Mumford over $\ZZ[1/2]$.
Moreover, there are $r$ and $d$ such that $\mathcal F_{r,d}$  is not Deligne-Mumford even over $\QQ$, e.g.~the stack $\mathcal F_{4,1}=B(\PGL_4)$ of Brauer-Severi threefolds. 
\end{remark}

\subsection{Separatedness of the stack of Fano threefolds}\label{section:separatedness}
We now study the separatedness of the stack of certain Fano threefolds over some dense open of $\Spec \ZZ$, and deduce consequences for finiteness of Isom-schemes and uniqueness of good models.
Note that some restriction on the index and degree are required here. For example, the stack $\mathcal{F}_{4,1}$ of Brauer-Severi threefolds is not separated.  
We will work with the following collection of indices and degrees.
\begin{equation} \label{def:Slist}
	\Slist = \{(1,2), (1,4), (1,6), (1,8), (1,14),  (2,1),(2,2),(2,3),(2,4)\}.
\end{equation}

\begin{lemma}[Separatedness]\label{lem:separatedness}
There exists an integer $N\geq 1$ such that for all $(r,d) \in \Slist$ the algebraic stack $\mathcal F_{r,d, \ZZ[1/N]}$ is separated over $\ZZ[1/N]$.
\end{lemma}
\begin{proof}
 Write  $\mathcal M = \mathcal F_{r,d}$, and note that $\mathcal M$ is a finite type algebraic stack over $\ZZ$ (Lemma \ref{lem:stack_Fid}). We claim that  $\mathcal M_\QQ$ is separated over $\QQ$. 
Benoist's results on the separatedness of the stack of smooth complete intersections \cite{Ben13} imply that $\mathcal M_\QQ$ is separated if $(r,d) \in \{(1,6), (1,8),(2,3),(2,4)\}$. In the remaining cases,  the Fano threefolds classified by $\mathcal M_\QQ$ are non-rational \cite[Table 12.2]{IskPro} (see \cite[Thm.~9.1.6]{IskPro} for more precise references). In particular, the stack $\mathcal M_{\QQ}$ is also separated over $\QQ$ by \cite[Cor.~1.3]{Ben13} (this being an application of \cite{MatMum}). As separatedness spreads out   \cite[Prop.~B.3.(xvii)]{Rydh2}, the result follows.
\end{proof} We use separatedness to deduce the following.

\begin{lemma}[Finite Isom-schemes]\label{lem:finite_isoms}
There exists an integer $N\geq 1$ such that for all schemes $B$ with $N$ invertible on $B$, for all $(r,d)\in \Slist$, and for all $X, Y\in \mathcal F_{r,d}(B)$, the morphism  $\Isom_B(X,Y)\to B$ is finite.
\end{lemma}
\begin{proof}
Write $\mathcal M = \mathcal F_{r,d}$, and let $N\geq 1$ be such that $\mathcal{M}_{\ZZ[1/N]}$ is separated over $\ZZ[1/N]$. Such an integer $N$ exists by Lemma \ref{lem:separatedness}. 

Let $B$ be a  scheme  such that $N$ is invertible on $B$. Let $X$ and $Y$ be objects of $\mathcal M(B)$. Since $\mathcal M_B$ is separated, its diagonal is proper \cite[Def. 7.6]{LMB}. In particular, as $\Isom_B(X,Y)\to B$ is obtained by base-change from the diagonal of $\mathcal M_B$, we see that $\Isom_B(X,Y)\to B$ is proper.   By Corollary \ref{cor:autom_groups_are_affine},  the morphism of schemes $\Isom_B(X,Y)\to B$  is affine. We conclude that  $\Isom_B(X,Y)\to B$ is finite  \cite[Lem.~3.3.17]{Liu2}.  
\end{proof}

\begin{lemma}[Good models are unique]\label{lem:unicity}
There exists an integer $N\geq 1$ such that for all integral regular noetherian schemes $B$ with function field $K$ and $N$ invertible on $B$, for all $(r,d)\in \Slist$, and for all $X$ and $Y$ in $\mathcal F_{r,d}(B)$, any isomorphism $X_K\to Y_K$ over $K$ extends uniquely to an isomorphism $X\to Y$ over $B$. 
\end{lemma}

\begin{proof} Let $N$ be as in Lemma \ref{lem:finite_isoms} and let $B$ be an integral regular noetherian scheme with function field $K$ on which $N$ is invertible. Let $X$ and $Y$ be objects of $\mathcal F_{r,d}(B)$ and suppose that $X_K$ and $Y_K$ are isomorphic. Then the scheme $I:=\Isom_B(X,Y)$ is finite over $B$ and has a $K$-point. Thus, by Zariski's main theorem (or more generally \cite[Prop.~6.2]{GLL}), the scheme $I$ has a unique $B$-point extending the given $K$-point. This concludes the proof. 
\end{proof}

\subsection{Fano threefolds of index $2$ and degree $5$}\label{section:index2_degree5}
 In Section \ref{section:separatedness}, we excluded Fano threefolds with Picard number $1$, index $2$ and degree $5$ (i.e.~$(2,5)\not\in \Slist$). The aim of this section is to prove the unicity of such Fano threefolds, \'{e}tale locally on the base scheme (Lemma \ref{lem:locally_isom}). This weaker analogue of Lemma \ref{lem:unicity} will turn out to be sufficient for our proof of the Shafarevich conjecture in this case.

We start with the following general criterion for a Fano variety to be infinitesimally rigid.   

\begin{lemma}\label{lem:vanishing_of_H1_abstract} Let $k$ be a field of characteristic zero, and let $n$, $\rho$, $r$ and $d$ be integers such that the set of $\bar{k}$-isomorphism classes of Fano $n$-folds with Picard number $\rho$, index $r$ and degree $d$ is finite. Then for all Fano $n$-folds $X$ over $k$ with Picard number $\rho$, index $r$ and degree $d$, we have $\mathrm{H}^1(X,\Theta_X)=0$.
\end{lemma}
\begin{proof} Note that $\mathrm{H}^1(X,\Theta_X) \otimes_k \bar{k} \cong \mathrm{H}^1(X_{\bar{k}}, \Theta_{X_{\bar{k}} } )$. Therefore, we may and do  assume that $k$ is algebraically closed.

We will prove the result using basic deformation theory. Assume for a contradiction that $\mathrm{H}^1(X,\Theta_X) \neq 0$.
As $k$ is of characteristic zero, $X$ is unobstructed (Lemma \ref{lem:smoothness}). Therefore, since the algebraic stack of Fano varieties satisfies Artin's axioms (and thus Axiom $(4)$ in \cite[Tag 07XJ]{stacks-project}), the non-triviality of $\mathrm{H}^1(X,\Theta_X)$ shows that there exist a smooth affine connected curve $T$ over $k$, a rational point $t\in T(k)$, and a smooth proper scheme $f: \mathcal X\to T$ whose fibre over $t$ is isomorphic to $X$, such that  the Kodaira-Spencer map 
\[ \Theta_{T} \to R^1 f_\ast \Theta_{\mathcal X/T}\] is non-zero. 
Due to the open nature of ampleness, on shrinking $T$, if necessary, we may further assume that $\mathcal{X}/T$ is a Fano scheme (Definition \ref{def:fano_scheme}). 

Note that, by Proposition \ref{prop:invariants}, all the fibres of $f$ have the same invariants as $X$. In particular, by our assumption on $n, \rho, r$ and $d$, there exists a Fano $n$-fold $Y$ over $k$ with Picard number $\rho$, index $r$ and degree $d$, and infinitely many closed points $p$ in $T$ such that $Y\cong \mathcal X_p$. Thus the morphism of schemes $\mathrm{Isom}_T(\mathcal X, Y \times_k T)\to T$ is dominant, as its image is an infinite (hence dense) subset of $T$.  (Here we use that a closed point $p$ with residue field $k(p) =k$ in $T$ lies in the image of $\mathrm{Isom}_T(\mathcal X, Y\times_k T)\to T$ if and only if $Y$ is $k$-isomorphic to $\mathcal X_p = \mathcal X \times_T \Spec k(p)$.)

Next, by Corollary \ref{cor:autom_groups_are_affine}, the morphism $\mathrm{Isom}_T(\mathcal X, Y\times_k T)\to T$ is of finite type. Thus by generic flatness and spreading out, replacing $T$ by a dense open if necessary, we may and do assume that $\mathrm{Isom}_T(\mathcal X, Y\times_k T)\to T$ is of finite type, flat and  dominant. Since flat finite type morphisms are open, replacing $T$ by a dense open if necessary, the morphism  $\mathrm{Isom}_T(\mathcal X, Y\times_k T)\to T$ is of finite type, flat and  surjective. 

As the Fano $n$-fold $\mathcal X\to T$ is trivial over the fppf cover $\mathrm{Isom}_T(\mathcal X, Y\times_k T)\to T$,   the Kodaira-Spencer map is zero.  
This contradiction proves the lemma.
\end{proof}

\begin{remark}
 Smooth Fano varieties $X$ over $\mathbb C$ with $\mathrm{H}^1(X,\Theta_X)~=~0$ may still admit
 non-trivial specialisations (see for example \cite[\S2.2]{PasPer}).
\end{remark}

The hypotheses of Lemma \ref{lem:vanishing_of_H1_abstract} hold for Fano threefolds of Picard number $1$, index $2$ and degree $5$, as we now explain.

\begin{lemma}\label{lem:classification_251}
Let $X$ and $Y$ be Fano threefolds over a field $k$. If $X$ and $Y$ have Picard number $1$, index $2$ and degree $5$, then  $X_{\bar k}$ and $Y_{\bar k}$ are isomorphic. 
\end{lemma}
\begin{proof}
	Over $\CC$ this is \cite[Cor~3.4.2]{IskPro}.  The general result follows from Megyesi's classification \cite[Thm.~6.(iii)]{Megyesi}. 
\end{proof}

From Lemma \ref{lem:vanishing_of_H1_abstract} and Lemma \ref{lem:classification_251} we immediately deduce the following.

\begin{corollary}\label{cor:vanishing_of_H1}  
 Let $X$ be a Fano threefold over a field $k$ of characteristic zero. If $\rho(X) = 1$, $\Index(X) =2$ and $\Degree(X) =5$, then $\mathrm{H}^1(X, \Theta_{X}) = 0$. 
\end{corollary}

We now use spreading out of smoothness and the boundedness of  the moduli of Fano threefolds with Picard number $1$, index $2$ and degree $5$ over $\mathbb Z$ to show that two such Fano threefolds over a suitable base are \'etale locally isomorphic.

\begin{lemma}\label{lem:isom_is_smooth}  There exists an integer $N\geq 1$ with the following property. Let $B$ be a scheme such that $N$ is invertible on $B$.
If $X$ and $Y$ are Fano threefolds  in $\mathcal F_{2,5}(B)$, then $\Isom_B(X,Y)\to B$ is   finite type, smooth and affine. 
\end{lemma}
\begin{proof}   Write $\mathcal M = \mathcal F_{2,5}$ for the stack of Fano threefolds with Picard number $1$, index $2$ and degree $5$ over $\mathbb Z$.  
Consider the diagonal $\Delta: \mathcal M \to \mathcal M \times_{\ZZ} \mathcal M$, and note that the stacks $\mathcal M$ and $\mathcal M \times_{\ZZ} \mathcal M$ are of finite type over $\ZZ$ (Lemma \ref{lem:stack_Fid}).  By Lemma \ref{lem:stack_Fid}, the representable morphism  $\Delta$ is  affine and of finite type.  Now, 
by Corollary \ref{cor:vanishing_of_H1} and \cite[Cor.~III.5.9]{SGA1} we find  that $\Delta_\QQ: \mathcal M_\QQ \to \mathcal M_\QQ \times_\QQ \mathcal M_\QQ$ is formally smooth, hence smooth. In particular, by applying spreading out of smoothness (\cite[Prop.~B.3.(xiii)]{Rydh2}) to the (finite type) diagonal morphism of $\mathcal M$, there exists an integer $N\geq 1$ such that $\Delta_{\ZZ[1/N]} : \mathcal M_{\ZZ[1/N]} \to \mathcal M_{\ZZ[1/N]} \times_{\ZZ[1/N]} \mathcal M_{\ZZ[1/N]} $ is smooth.  Hence for all schemes $B$ over $\ZZ[1/N]$ and for all $X$ and $Y$ in $\mathcal F_{2,5}(B)$, the morphism $\Isom_B(X,Y)\to B$, obtained via base-change from $\Delta_{\ZZ[1/N]}$, is  finite type, smooth and affine.
\end{proof}

\begin{lemma}\label{lem:locally_isom} There exists an integer $N\geq 1$ with the following property. Let $B$ be a scheme such that $N$ is invertible on $B$.
If $X$ and $Y$ are Fano threefolds in $\mathcal F_{2,5}(B)$, then $X$ and $Y$ are locally isomorphic over $B$ for the \'etale topology, i.e.~there exists an \'etale finite type surjective morphism $C\to B$ such that $X_C$ and $Y_C$ are isomorphic over $C$.
\end{lemma}
\begin{proof} Choose $N$ as in Lemma \ref{lem:isom_is_smooth}. Write $I:=\Isom_B(X,Y)$. Note that it suffices to show that $I\to B$ has a section locally for the \'etale topology on $B$.
 By Lemma \ref{lem:classification_251}, the morphism $ I\to B$ has non-empty geometric fibres, and is therefore surjective. By Lemma \ref{lem:isom_is_smooth},  the morphism $I\to B$ is smooth and of finite type. As smooth surjective finite type morphisms of schemes have sections locally for the \'etale topology, this concludes the proof.
\end{proof}

\subsection{Period maps of Fano threefolds}\label{section:periodmap}
We now study the period maps of certain Fano threefolds over $\CC$, in particular Torelli type problems (see \cite{Catanese} for a discussion of some Torelli type problems).
Due to a lack of global Torelli theorems, we will settle for infinitesimal Torelli theorems. Note that numerous classes of Fano threefolds over $\CC$ are known to \emph{not} satisfy the infinitesimal Torelli property; for instance,~Fano threefolds of Picard number $1$, index $1$ and degree $10$ \cite[Thm.~7.4]{DIM} or degree $14$ \cite[Thm.~5.8]{IM}.

\subsubsection{Infinitesimal Torelli}
We first show how to deduce the infinitesimal Torelli theorem in certain cases from known results in the literature, namely, for Fano threefolds with the following invariants. 
\begin{equation} \label{def:Tlist}
	\Tlist = \{(1,2), (1,6), (1,8), (2,1), (2,2), (2,3), (2,4)\}.
\end{equation}
Note that $\Tlist \subset \Slist$ in the notation of \eqref{def:Slist}.

\begin{proposition}\label{prop:inf_Torelli}
Let $X$ be a Fano threefold over $\CC$ with $\rho(X) = 1$ and   $(\Index(X),\Degree(X)) \in \Tlist$.
Then $X$ satisfies the infinitesimal Torelli property, i.e.~the morphism $$\mathrm{H}^1(X,\Theta_X) \to \Hom(\mathrm{H}^{2,1}(X),\mathrm{H}^{1,2}(X)) $$ is injective.
\end{proposition}

\begin{proof}  
Let $X$ be as in the statement of the proposition. If $(\Index(X),\Degree(X)) = (1,2), (2,1)$ or $(2,2)$, then $X$ is a certain hypersurface in some weighted projective space. The infinitesimal Torelli theorem in these cases follows from work of Saito, namely \cite[Thm.~7.3]{Sai86} and \cite[Thm.~7.6]{Sai86}. (Note that Saito, as many other authors, refers to   ``infinitesimal Torelli'' as   ``local Torelli''; see \cite[Def.~7.2]{Sai86}.) The remaining cases are certain smooth complete intersections in projective space, where the result follows   from Flenner's infinitesimal Torelli theorem \cite[Thm.~3.1]{Fl86}.
\end{proof}

\subsubsection{Quasi-finite Torelli for some Fano threefolds} A period map (which satisfies Griffiths transversality) on a quasi-projective complex manifold with injective differential has finite fibres. We record  this finiteness statement for the  Fano threefolds that interest us below (Theorem \ref{thm:qf_torelli}).
 
 To state the next result, for X a finite type $\mathbb C$-scheme, we let $X^{\an}$ be the associated complex analytic space \cite[Exp.~XII]{SGA1}.
  
\begin{theorem}\label{thm:qf_torelli}
Let $H$ be a polarised $\ZZ$-Hodge structure. Then the set of isomorphism classes of Fano threefolds $X$ over $\CC$ such that 
\begin{enumerate} 
\item $\rho(X) =1$, 
\item $(\Index(X), \Degree(X)) \in \Tlist$, and 
\item $H$ isomorphic to the polarised $\ZZ$-Hodge structure $\mathrm{H}^3(X^{\an},\ZZ)$
\end{enumerate} is finite.
\end{theorem}
\begin{proof}  Let $(r,d) \in \Tlist$ and let $X$ be a Fano threefold over $\CC$ with $\rho(X) =1$, $\Index(X) = r$ and $\Degree(X) =d$. By Lemma~\ref{lem:stack_Fid} and Lemma~\ref{lem:finite_isoms}, the algebraic stack $\mathcal F_{r,d,\CC}$ is smooth finite type  and Deligne-Mumford. In particular, there exist a smooth affine     $\CC$-scheme $U$ and an \'etale surjective morphism $U\to \mathcal F_{r,d,\CC}$ over $\CC$. Let $f:Y\to U$ be the pull-back of the universal family over $\mathcal F_{r,d,\CC}$. 

 Note that the set of $u$ in $U(\mathbb C)$ with $Y_u \cong X$ is finite and non-empty. Consider the period map $p:U^{\an}\to \Gamma\backslash D$ associated to $R^3f_\ast \ZZ$ and the choice of a base-point $u$ in $U(\CC)$. Here $D$ is the  period domain defined by the polarised $\ZZ$-Hodge structure $\mathrm{H}^3(Y_u, \ZZ)$ and $\Gamma$ is the monodromy group of the family of Fano threefolds $f:Y\to U$; see \cite[Sect.~4.3-4.4]{Mullerstach} or \cite[Ch.~10]{Voi07} for a detailed treatment of the construction of $D$ and $p$. 
To prove the theorem, it suffices to show that $p$ has finite fibres. 

By Selberg's lemma (see  \cite[Thm.~II]{Cassels} or \cite[Lem.~8]{Selberg}), replacing $U$ by an \'etale covering if necessary, we may and do assume that $\Gamma$ acts freely on $D$, so that, by Proposition \ref{prop:inf_Torelli}, the period map $p$ is an immersion. 

We now show that $p:U^{\an}\to \Gamma\backslash D$ has finite fibres, using a similar method to the proof of \cite[Thm.~2.8]{JL}. 
By \cite[p.~122]{G}   or \cite[Cor.~13.4.6]{Mullerstach}, there exist a smooth quasi-projective  scheme  $U^\prime$ over $\CC$, an open immersion $U\to U^\prime$ of schemes, and a proper morphism of complex analytic spaces $p^ \prime: U^{\prime,\an}\to D/\Gamma$ which extends the period map $p:U^{\an}\to D/\Gamma$. An application of Stein factorisation for proper morphisms of complex analytic spaces \cite[Ch.~10.6.1]{GrauertRemmert} shows that there exist a proper surjective morphism  $p_0:U^{\prime,\an}\to U_0$ with connected fibres, and a finite morphism  $U_0\to D/\Gamma$ such that $p^\prime$ factorises as \[\xymatrix{ U^{\prime,\an} \ar[rr]^{ p^\prime} \ar[dr]_{p_0} & & D/\Gamma \\ & U_0. \ar[ur] & } \]  
Since the restriction of $p^\prime$ to $U^{\an}$ is an immersion, it follows that $p_0$ is an isomorphism
when restricted to $U^{\an}$. In particular, the period map $p$ factors as $U^{\an} \subset U_0 \to  D/\Gamma$,
hence has finite fibres.  
 \end{proof}

\section{Good reduction}\label{section:finiteness}

\subsection{Fano threefolds}

In this section we prove Theorems \ref{theorem1}, \ref{theorem2} and \ref{theorem3}. 
We begin with the following definition.

\begin{definition} \label{def:smooth_reduction}
Let $B$ be an integral noetherian scheme with function field $K$ and let $X$ be
 a proper variety over $K$. A \emph{model for $X$ over $B$} is a flat proper $B$-scheme
 $\mathcal X\to B$ together with a choice of isomorphism $\mathcal{X}_K \cong X$.
 We say that
 \begin{enumerate}
\item 	 $X$ has \textit{smooth reduction at a point $v$ of $B$} if $X$ has a smooth model over the localisation
	$B_v$ of $B$ at $v$.
\item  $X$ has \textit{smooth reduction over $B$} if
	$X$ has smooth reduction at all points of codimension one of $B$.
 \end{enumerate} 
\end{definition}

Our first result is key, and shows the finiteness of the set of ``periods'' (i.e.~integral polarised Hodge structures) of Fano threefolds with smooth reduction. It makes use of the theory of the intermediate Jacobian over non-algebraically closed fields, due to Achter--Casalaina-Martin--Vial \cite{ACM}, together with Faltings's finiteness results for abelian varieties \cite{Faltings2, FaltingsComplements, Szpiroa}. 
 
\begin{proposition}\label{prop:finiteness_of_hs} Let $B$ be an arithmetic scheme with function field $K$. Let $\sigma:K\to \mathbb C$ be an embedding.
Then the set of isomorphism classes of polarised $\mathbb Z$-Hodge structures $\mathrm{H}^3(X_{\mathbb C},\mathbb Z)$, where $X$ runs over all
Fano threefolds over $K$ with smooth reduction over $B$, is finite.
\end{proposition}

\begin{proof} 
For $X$ a  Fano threefold over $K$, we let $J(X)$ be its intermediate Jacobian over $K$. The existence of this principally polarised abelian variety over $K$ is a special case of the main result of \cite{ACM} (cf.~\cite[Cor.~5.6]{ACM}). The intermediate Jacobian $J(X)$ has the following cohomological properties.
\begin{enumerate}
 \item $\mathrm{H}^1(J(X)_{\bar{K}},\QQ_\ell) \cong \mathrm{H}^3(X_{\bar{K}},\QQ_\ell(1))$ as $\Gal(\bar{K}/K)$-representations. \label{eqn:QQ_ell}
 \item $\mathrm{H}^1(J(X_{\CC}),\ZZ) \cong \mathrm{H}^3(X_{\CC}, \ZZ(1))$ as polarised $\ZZ$-Hodge structure. 
 \label{eqn:ZZ}
\end{enumerate}

As there are only finitely many deformation types of complex algebraic Fano threefolds (see  \cite[\S 5]{Deb97}, \cite[Thm.~5.19]{Debarrebook} or \cite{IskPro}), the intermediate Jacobian $J(X)$  has bounded dimension. Moreover, if $X$ has smooth reduction over $B$, then property \eqref{eqn:QQ_ell} together with the N\'eron-Ogg-Shafarevich criterion  \cite{SerreTate} shows that $J(X)$ also has smooth reduction over $B$.  
Thus, by Faltings's finiteness theorem \cite{FaltingsComplements}, the set of $K$-isomorphism classes of intermediate Jacobians $J(X)$,  where $X$ runs over all Fano threefolds   over $K$ with smooth reduction over $B$, is finite. In particular, the set of isomorphism classes of polarised integral Hodge structures   $\mathrm{H}^1(J(X_{\CC}),\ZZ)$, for such Fano threefolds $X$, is finite. The result then follows from property \eqref{eqn:ZZ}.
\end{proof}

 \begin{definition} \label{def:good_reduction}
 Let $B$ be an integral noetherian scheme with function field $K$. Let $X$ be a Fano variety over $K$. A \emph{good model for $X$ over $B$} is a Fano scheme $\mathcal X\to B$ together with a choice of isomorphism $\mathcal{X}_K \cong X$. We say that $X$ has \emph{good reduction over $B$} if, for all $v $ in $B^{(1)}$, the Fano variety $X$ has a good model over $\Spec \mathcal O_{B,v}$.
 \end{definition}

\begin{lemma}\label{lem:good_is_good1} Let $B$ be an integral noetherian scheme with function field $K$. Let $X$ be a Fano variety over $K$. Then $X$ has good reduction over $B$ if and only if there exists a dense open subscheme $U\subset X$ with
$X\setminus U$ of codimension at least two such that $X$ has a good model over $U$.
\end{lemma}
\begin{proof}
This follows from a standard argument (see e.g.~\cite[Prop.~1.4.1]{BLR}), by gluing models over suitable open neighbourhoods of $B$ along their generic fibre.
\end{proof}

Over Dedekind schemes there is no difference between good reduction and having a good model.

\begin{lemma}\label{lem:good_is_good} Let $B$ be a Dedekind scheme with function field $K$. Let $X$ be a Fano variety over $K$. Then $X$ has good reduction over $B$ if and only if $X$ has a good model over $B$.
\end{lemma}
\begin{proof} This follows from Lemma \ref{lem:good_is_good1}.
 \end{proof}

\subsubsection{Twists}
We now prove some properties about good reduction and twists.

\begin{definition}
Let $X$ be a Fano variety over a field $k$. A Fano variety $Y$ over $k$ is a \emph{twist of $X$ over $k$} if $X_{\bar k}$ is isomorphic to $Y_{\bar k}$ with respect to some algebraic closure $k\to \bar{k}$.
\end{definition} 

Our next finiteness result uses the finiteness result of Hermite-Minkowski and the separatedness of certain    stacks of Fano threefolds (Section \ref{section:separatedness}). Analogous finiteness results have been proven for varieties with ample canonical bundle \cite[Lem.~4.1]{Jav15} and complete intersections \cite[Thm.~4.10]{JL}. Recall the definition of $\Slist$ given in \eqref{def:Slist}.

\begin{proposition}\label{prop:twists} Let $B$ be an arithmetic scheme with function field $K$. Let $X$ be a Fano threefold over $K$ with $\rho(X) =1$ and  $(\Index(X_K),\Degree(X_K)) \in \Slist$. Then the set of $K$-isomorphism classes of Fano threefolds $Y$ over $K$ which are twists of $X$ and have good reduction over $B$ is finite.
\end{proposition}
\begin{proof}  
In what follows, we are allowed to shrink $B$ if necessary, i.e.~replace $B$ by a dense open subscheme. Thus, choosing $N\geq 1$ as in Lemma \ref{lem:finite_isoms} and Lemma \ref{lem:unicity}, we may assume that $N$ is invertible on $B$.  Further shrinking $B$, if necessary, we may assume that $X$ has a good model over $B$, say $\mathcal X\to B$. Again shrinking $B$, we may assume that $\Aut_B(\mathcal X)$ is finite \'etale over $B$. Indeed, the automorphism group scheme $\Aut_K(X)$ of $X$ is finite \'etale over $K$ by our assumptions on the index and degree of $X$  (Lemma \ref{lem:finite_isoms}), and $\Aut_B(\mathcal X)\to B$ is of finite type over  $B$ (Corollary \ref{cor:autom_groups_are_affine}), so that we may apply spreading out for finite \'etale morphisms. 

Let $Y$ be a twist of $X$ over $K$ with a good reduction over $B$. For each $v \in B^{(1)}$, let $\mathcal Y_v$ be a good model for $Y$ over $B_v = \Spec \OO_{B,v}$. Note that Lemma \ref{lem:openness} implies  that $\mathcal Y_v$ has Picard number $1$, index $\Index(X)$ and degree $\Degree(X)$.  In particular, we deduce from Lemma \ref{lem:unicity} the following: Let $L/K$ be a finite field extension such that $X_L$ and $Y_L$ are isomorphic over $L$. Then for all $v\in B^{(1)}$ we have $\mathcal X \times_{B_v} C_v \cong \mathcal Y_v\times_{B_v} C_v$, where $C_v\to B_v$ is the normalization of $B_v$ in $L$. Since $C_v\to B_v$ is finite flat surjective and the affine finite type scheme $\Isom_{B_v}(X_v,Y_v)$ has a $C_{v}$-point,   we conclude that  $\Isom_{B_v}(X_v,Y_v)$ is an $\Aut_{B_v}(\mathcal X_v)$-torsor over $B_v$.
Hence the class $[Y]$ of $Y$ in $\mathrm{H}^1(K,\Aut_K(X))$  lies in the subset
$$ 
\bigcap_{v \in B^{(1)}}\im\left( \mathrm{H}^1(B_v,\Aut_{B_v}(\mathcal X_v)) \to \mathrm{H}^1(K, \Aut_K(X)) \right).$$
That this set is finite now follows from Hermite-Minkowski for arithmetic schemes (see \cite[Lem.~4.5]{JL}). The result is proved. 
\end{proof}

\subsubsection{Finiteness results}

Recall the definition of $\Tlist$ given in \eqref{def:Tlist}.

\begin{theorem}\label{thm:finiteness_for_fanos}
Let $B$ be an arithmetic scheme with function field $K$. The set of $K$-isomorphism classes of Fano threefolds $X$ over $K$ with Picard number $1$, good reduction over $B$ and $(\Index(X),\Degree(X) ) \in \Tlist $ is finite.
\end{theorem}
\begin{proof} Fix an embedding $K\to \CC$. By the  finiteness of periods (Proposition \ref{prop:finiteness_of_hs}) and the quasi-finite Torelli  (Theorem \ref{thm:qf_torelli}),  the set of $\CC$-isomorphism classes of Fano threefolds $X$ over $K$ with Picard number $1$, good reduction over $B$ and $(\Index(X),\Degree(X) ) \in  \Tlist$ is finite. The result then follows from Proposition~\ref{prop:twists}.
\end{proof}

\begin{remark}\label{remark:lang_vojta} We give an (imprecise) explanation of how Theorem \ref{thm:finiteness_for_fanos} is compatible with the Lang-Vojta conjecture; see \cite{Abr, CHM}, \cite[Conj.~F.5.3.6]{HindrySilverman}, and \cite[Conj.~6.1]{JL}. This can be made more precise using the techniques from \cite[\S 6]{JL}.

Let $B$ be a smooth quasi-projective   scheme over $\CC$, and let $  X\to B$ be a non-isotrivial Fano threefold in $\mathcal F_{r,d}(B)$ with~$(r,d) \in \Tlist.$
The infinitesimal Torelli theorem (Proposition \ref{prop:inf_Torelli}) and  properties of  period maps with injective differential (see  \cite[Lem.~6.3]{JL}) imply that $B$ is in fact Brody hyperbolic, i.e.~all  holomorphic maps $\CC\to B(\CC)$ are constant.  In particular, the   stack $\mathcal F_{r,d,\CC}$ is Brody hyperbolic. Ignoring stacky issues, the Lang-Vojta conjecture predicts that the moduli stack of such Fano threefolds has only finitely many (isomorphism classes of) integral points, which Theorem \ref{thm:finiteness_for_fanos} indeed confirms.
\end{remark}

\begin{proposition} \label{prop:125}
Let $K$ be  a number field and let $S$ be a finite set of finite places of $K$. The set of $K$-isomorphism classes of Fano threefolds $X$ over $K$ with Picard number $1$, index $2$, degree $5$ and good reduction outside of $S$ is finite.
\end{proposition}
\begin{proof} Choose $N$ as in Lemma \ref{lem:locally_isom}. We may assume that $S$ contains all the places of $K$ lying over the prime divisors of $N$. Write $B=\Spec \OO_K[S^{-1}]$.   By Lemma \ref{lem:openness} and Lemma \ref{lem:good_is_good}, it suffices to show that the set of $B$-isomorphism classes of Fano threefolds $X\to B$ of Picard number $1$, index $2$ and degree $5$ is finite. 
 
 Let $X$ be such a Fano threefold over $B$. By Lemma \ref{lem:locally_isom} the set of $B$-isomorphism classes of   Fano threefolds over $B$ of Picard number $1$, index $2$ and degree $5$ is a subset of $\mathrm{H}^1 (B,\Aut_B(X))$. As $\Aut_B(X)$ is affine and of finite type (Corollary \ref{cor:autom_groups_are_affine}), this cohomology set is finite by \cite[Prop.~5.1]{GilleMoretBailly}. This concludes the proof.
\end{proof}

\begin{proof}[Proof of Theorem \ref{theorem1}] Let $B=\Spec \OO_K[S^{-1}]$ and recall the classification given in Section \ref{section:clas}.   Since Fano threefolds with Picard number $1$ and  index $>2$ are flag varieties, the set of such Fano threefolds  with good reduction over $B$ is finite \cite[Thm.~1.4]{JL2}. The remaining cases follow from Theorem \ref{thm:finiteness_for_fanos} and Proposition \ref{prop:125}. 
\end{proof}

\begin{proof}[Proof of Theorem \ref{theorem2}]
This follows immediately from Theorem \ref{thm:finiteness_for_fanos}.
\end{proof}

\subsection{Sextic surfaces} \label{sec:sextic}

In this section we give an application of our results to smooth sextic surfaces. Here we use the following notion of good reduction, which agrees with the notion of good reduction used in \cite{JL}.
 \begin{definition} Let $B$ be an integral noetherian scheme with function field $K$. Let $X$ be a smooth sextic surface over $K$. 
 
A \emph{good model} for $X$ over $B$ is a smooth sextic surface over $B$ together with a choice of isomorphism $\mathcal{X}_K \cong X$. We say that $X$ has \emph{good reduction} over $B$ if, for all $v \in B^{(1)}$, the sextic $X$ has a good model over $\Spec \mathcal O_{B,v}$.
 \end{definition}

Here, by a smooth sextic surface over $B$, we mean a smooth proper scheme over $B$ in $\PP^{3}_B$ which is the zero locus of a non-zero global section of the sheaf $\pi_* \OO_{\PP^{3}_B}(6)$, where $\pi: \PP^{3}_B \to B$ denotes the natural projection.

We will deduce Theorem \ref{theorem3} from the following more general result.

\begin{theorem}\label{thm:sextics}
Let $B$ be an arithmetic scheme. Then the set of $B$-isomorphism classes of smooth sextic surfaces in $\mathbb P^3_B$ is finite.
\end{theorem}
\begin{proof} 
We  prove the result by combining Theorem \ref{thm:finiteness_for_fanos} with some of the tools from our paper on complete intersections \cite{JL}. Let $K$ be the function field of $B$ and let $K\to \bar K$ be an algebraic closure. By the unicity of good models \cite[Lem.~4.7]{JL}, it suffices to show that the set of $K$-isomorphism classes of smooth sextic surfaces which have a good model over $B$ is finite. In particular, we may assume that $B=\Spec A$ is affine and that $6$ is invertible in $A$, so that a good model is defined by some homogeneous equation of degree $6$ over $A$. 

Let $X$ be a smooth sextic surface over $K$ with a good model $\mathcal X_f$ in $\mathbb P^3_B$, defined by some homogeneous polynomial $f$ of degree $6$ over $A$. 
Consider the scheme $\mathcal{Y}_f$ in $\PP_B(3,1,1,1,1)$ over $B$ defined by $y^2 = f(x)$. This is a Fano threefold over $B$ of Picard number $1$, index $1$ and genus $2$ (see Section \ref{section:clas}). Note that this construction depends on the choice of good model for $X$; different choices may give rise to non-isomorphic $\mathcal{Y}_f$ (e.g.~simply multiply the equation of $f$ by a non-square in the unit group $A^*$). Nevertheless, by Theorem \ref{thm:finiteness_for_fanos}, we see that as $\mathcal{X}_f$ varies over all possible good models for each such smooth sextic $X$, we obtain only finitely many $B$-isomorphism classes for the $\mathcal{Y}_f$. 
The required finiteness now easily follows.
\end{proof}

\begin{proof}[Proof of Theorem \ref{theorem3}] Write $B=\Spec \OO_K[S^{-1}]$. Shrinking $B$ if necessary, 
we may assume that $\Pic B =0$. Hence  any smooth sextic surface $X$ over $K$ with good reduction over $B$ has a good model over $B$ by  \cite[Lem.~4.8]{JL} (this is an analogue of Lemma \ref{lem:good_is_good} for complete intersections). The result then follows from Theorem \ref{thm:sextics}.
\end{proof}

\begin{remark}\label{remark:phil}
The idea of the proof of Theorem \ref{thm:sextics} can be captured by the following ``correspondence'' of complex algebraic stacks
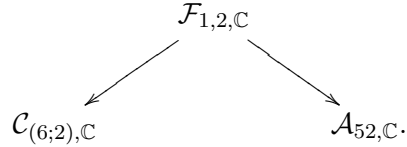
\[\xymatrix{ & \mathcal F_{1,2,\CC} \ar[dl] \ar[dr] & \\ \mathcal C_{(6;2),\CC} & & \mathcal A_{52,\CC}.} \]
Here $\smash{\mathcal F_{1,2,\CC}\to \mathcal C_{(6;2),\CC}}$ is the $\mu_2$-gerbe which associates to a Fano threefold in $\mathcal F_{1,2,\CC}$ the branch locus of the associated anticanonical double covering of $\PP^3$, and  $\mathcal F_{1,2,\CC}\to \mathcal A_{52,\CC}$ is the quasi-finite morphism of algebraic stacks which associates to a Fano threefold in $\mathcal F_{1,2,\CC}$ its intermediate Jacobian; the algebraicity of the intermediate Jacobian as a morphism of analytic stacks follows from  Borel's theorem (see \cite[Thm.~3.10]{Borel1972} or \cite[Thm.~5.1]{DeligneK3}). It is this correspondence which allows to relate sextic surfaces to abelian varieties, and, in particular, allows us to prove the Shafarevich conjecture for sextic surfaces using the Shafarevich conjecture for abelian varieties.
\end{remark}

\section{Counter-examples: infinitude and non-density} \label{section:CE}
The aim of this section is to prove Theorem \ref{thm:CE} and Theorem \ref{thm:non_density1}. To do this, we shall consider the class of Fano varieties obtained by blowing up lines on smooth complete intersections of two quadrics in $\PP^5$. These threefolds appear as No.~19 of Table 12.3 of \cite{IskPro} in the classification of Fano threefolds, and have index $1$, degree $26$ and third Betti number $4$. 

We will require certain results on such threefolds proved in \cite{IskPro}. However, the authors of \emph{loc.~cit}.~work throughout over a field of characteristic $0$. Nonetheless, one can check that the results from \emph{loc.~cit}.~that we require in this section continue to hold over fields of characteristic not equal to $2$.

\subsection{A counter-example}\label{section:infinite}
Let $k$ be a field of characteristic not equal to $2$
and let $X$ be a smooth complete intersection of two quadrics in $\PP^5_k$.
The lines in $X$ are parametrised by a smooth proper surface over $k$, which is a torsor
under the intermediate Jacobian $J(X)$ of $X$ (see e.g.~\cite[Thm.~17.0.1]{CF96}).
Let $L \subset X$ be a line on $X$ (if it exists). The blow-up $\Bl_L X$ of $X$ at $L$ is a Fano threefold
of Picard number $2$ \cite[Prop.~3.4.1]{IskPro}.

\begin{lemma} \label{lem:E}
	Let $X$ be a smooth complete intersection of two quadrics in $\PP_k^5$ which contains
	a line $L \subset X$. Denote by $\pi:\Bl_L X \to X$ the blow-up of $X$ at $L$ with exceptional divisor $E$.
	Then $E$ is the unique non-trivial reduced divisor on $X$ with the property $E^3 = 0$.
\end{lemma}
\begin{proof}
	The Picard group of $\Bl_L X$ is the free abelian group generated by $E$ and $H^*:=\pi^* H$, where $H$ denotes the class
	of the hyperplane section on $X$. By \cite[Lem.~2.2.14]{IskPro} we have
	$$(H^*)^3 = 4, \, (H^*)^2 \cdot E = 0, \, (H^*) \cdot E^2 = -1, \, E^3 = 0.$$
	It follows that for any class $a H^* + b E \in \Pic (\Bl_L X)$ we have
	$$(aH^* + b E)^3 = 4a^3 - 3ab^2 = a(4a^2 - 3b^2).$$
	Hence for the above intersection number to be $0$, we must have $a = 0$.
	However, if $f$ denotes the class in the Chow group $A^2(\Bl_L X)$ of a fibre of the ruling of $E$,
	then by \cite[Lem.~2.2.14]{IskPro} we have $E \cdot f = -1$. It follows that $\mathrm{H}^0(\Bl_L X, bE) = k$
	for all $b > 0$, which proves the result. 
\end{proof}

We now study the isomorphism classes of these varieties. Let $X_1$ and $X_2$
be smooth complete intersections of two quadrics in $\PP_k^5$ which each
contain a line  $L_i \subset X_i$.
Denote by $\Isom( (X_1,L_1), (X_2, L_2))$ the scheme of isomorphisms from $X_1$ to $X_2$ which map $L_1$ to $L_2$.
From the universal property of blow-ups \cite[Prop.~7.14]{Har77}, we obtain a natural morphism of schemes
\begin{equation} \label{def:natural}
	\Isom( (X_1,L_1), (X_2, L_2)) \to \Isom( \Bl_{L_1} X_1, \Bl_{L_2} X_2).
\end{equation}

\begin{lemma} \label{lem:iso_L}
	The morphism \eqref{def:natural} is an isomorphism.
\end{lemma}
\begin{proof}
	As $2\in k^\ast$, by \cite[Thm.~3.1~and~Prop.~3.7]{Ben13} we have $\mathrm{H}^0(X_i,\Theta_{X_i})=0$. It follows that
	the schemes $\Isom( (X_1,L_1), (X_2, L_2)) $ and $\Isom( \Bl_{L_1} X_1, \Bl_{L_2} X_2)$ are both smooth over $k$.
	Moreover $\Isom( (X_1,L_1), (X_2, L_2)) $ is finite over $k$ by Lemma \ref{lem:finite_isoms}. Hence it
	suffices to show that  \eqref{def:natural} is a bijection on geometric points.
	
	It is clearly injective. To show surjectivity, consider 
	$f \in \Isom( \Bl_{L_1} X_1, \Bl_{L_2} X_2)$. Denote by $\pi_i$ the corresponding blow-up map,
	by $E_i$ the exceptional divisor and by $H_i^*:=\pi_i^* H$. By Lemma~\ref{lem:E}
	we have $f(E_1)= E_2$. Moreover, clearly $f^*K_{\Bl_{L_2} X_2} = K_{\Bl_{L_1} X_1}$. Also,
	we have $-K_{\Bl_{L_i} X_i} = -\pi_i^* K_{X_i}  - E_i= 2H_i^* - E_i$. 
	Therefore, as $\Pic(\Bl_{L_1}X_1)$ is a free abelian group, we find that
	$f^*H_2^* = H_1^*$. Thus, there exists an automorphism $g:\PP^5_k \to \PP^5_k$
	such that the following diagram
	$$
	\xymatrix{\Bl_{L_1} X_1 \ar[r]^f \ar[d]^{\pi_1} &	\Bl_{L_2} X_2 \ar[d]^{\pi_2} \\
	X_1 \ar@{^{(}->}[d] & X_2 \ar@{^{(}->}[d] \\
	\PP^5_k \ar[r]^{g} & \PP^5_k }
	$$ commutes. The result follows.
\end{proof}

We now prove the following more precise version of Theorem \ref{thm:CE}.

\begin{theorem}\label{theorem:counter-example}
The set of $\Qbar$-isomorphism classes of Fano threefolds $X$ over $\mathbb Q$ with Picard number $2$, index $1$, degree $26$, third Betti number $4$, and good reduction outside away from $2$ and $29$ is infinite.
\end{theorem}
\begin{proof}
Let $X$ be a smooth complete intersection of two quadrics   in $\PP^5_{\QQ}$ 
which contains infinitely many lines over $\QQ$, and has good reduction away from $S:=\{2,29\}$;  the example given in \cite[Ex.~3.6]{BL15} satisfies these properties.  Choose a good model $\mathcal{X} \subset \PP^5_B$
for $X$ over $B:=\ZZ[S^{-1}]$; this exists by \cite[Lem.~4.8]{JL}. Let $L \subset X$ be a line. As the relative scheme of lines $\mathcal{L}(\mathcal{X}/B)$ is proper over $B$, there exists a unique relative line $\mathscr L$ on $\mathcal X$ such that $\mathscr L_\QQ =L$. 
As the blow-up $\Bl_\mathscr{L} \mathcal{X}$ is a good model for $\Bl_L X$ over $B$,  the Fano threefold $\Bl_L X$ has good reduction over $B$.
However, by Lemma \ref{lem:finite_isoms} we know that $\Aut(X)$ is finite. A simple application of 
Lemma \ref{lem:iso_L} therefore shows that as $L$ runs over the infinitely many lines of $X$ defined over $\QQ$,
we obtain infinitely many $\bar{\QQ}$-isomorphism classes for the $\Bl_L X$. This gives the required counter-examples.
\end{proof}

\begin{proof}[Proof of Theorem \ref{thm:CE}] This clearly follows from Theorem \ref{theorem:counter-example}.
\end{proof}

\begin{remark}
Note that the Fano threefolds constructed in the proof of Theorem \ref{theorem:counter-example} do not satisfy the infinitesimal Torelli theorem. Indeed, they each have the same intermediate Jacobian, as they differ by the blow-up of a curve of genus zero on a fixed smooth intersection of two quadrics in $\PP^5_\CC$ (see \cite[Lem.~3.11]{CG72}). The global Torelli theorem for such intersections \cite[Cor.~3.4]{Don80} shows that the period map for the Fano threefolds appearing in the proof of Theorem \ref{theorem:counter-example} has two-dimensional smooth proper fibres. 
\end{remark}

\begin{remark}
	Our method does not allow us to take $S = \emptyset$ in the statement of Theorem \ref{thm:CE}. 
	In our proof of Theorem \ref{theorem:counter-example} 
	we require $S$ to contain all primes of bad reduction for $X$,
	however, a simple application of the theorem of Abrashkin \cite{Abrashkin} and Fontaine \cite{Fontaine} 
	shows that there is no smooth odd-dimensional complete intersection
	of two quadrics over $\QQ$ with everywhere good reduction.
\end{remark}

\subsection{Non-density of integral points} Although the set of (isomorphism classes of) $\ZZ[1/58]$-integral points on the stack of Fano threefolds with Picard number $2$, index $1$, degree $26$ and third Betti number $4$ is infinite (by Theorem \ref{theorem:counter-example}), we now show that its $\OO_K[S^{-1}]$-points are never dense (see Theorem \ref{thm:non_density2}).

\subsubsection{Lines on intersections of two quadrics}
Let $\mathcal C:=\mathcal C_{(2,2;3)} = \mathcal{F}_{2,4}$ be the stack of smooth intersections of two quadrics in $\mathbb P^5$ (see \cite{Ben13,JL}). Note that $\mathcal C$ is a smooth finite type separated algebraic stack over $\ZZ$, and that $\mathcal C_{\ZZ[1/2]}$ is Deligne-Mumford over $\ZZ[1/2]$ (see \cite[Thm.~1.6]{Ben13}). Let $\mathcal U\to \mathcal C$ be the universal family, and let $\mathcal L\to \mathcal C$ be the stack of lines on $\mathcal U\to \mathcal C$. More precisely, for $S$ a scheme, the groupoid $\mathcal L(S)$ has objects pairs $(X\to S, L)$ with $X\to S$ in $\mathcal C(S)$ and $L$ a line on $X$ over $S$, and a morphism from $(X\to S,L)$ to $(X'\to S',L')$ in $\mathcal L$ is a morphism from $X\to S$ to $X'\to S'$ in $\mathcal C$ which sends the line $L$ to   the line $L'$.

\begin{lemma}\label{lem:smoothproper} The forgetful morphism $\mathcal L\to \mathcal C$ is representable by schemes, smooth and proper.  
\end{lemma}
\begin{proof}  
By the theory of Hilbert schemes, the morphism of stacks $\mathcal L\to \mathcal C$ is representable by schemes and proper. That it is smooth is a well-known property of intersections of two quadrics \cite[Prop.~3.3.5(i)]{IskPro}.
\end{proof}

\subsubsection{A stack of Fano threefolds}
Let $\mathcal M$ be the stack of Fano threefolds with Picard number $2$, index $1$,  degree $26$ and third Betti number $4$. It follows from Lemma \ref{lem:algebraic} and Lemma \ref{lem:openness} that $\mathcal M$ is a locally finite type algebraic stack over $\ZZ$ with affine finite type diagonal.  
There is a  morphism
\begin{equation}\label{morphism}
\mathcal L_{\ZZ[1/2]}\to \mathcal M_{\ZZ[1/2]}
\end{equation}  given by sending a pair $(X\to S,L)$ to the blow-up of $X$ at $L$. 
\begin{lemma}\label{lem:isom_of_stacks}
 The morphism $\mathcal L_{\ZZ[1/2]} \to \mathcal M_{\ZZ[1/2]}$ in \eqref{morphism} is an isomorphism of stacks.
\end{lemma}
\begin{proof}
This follows from  Lemma \ref{lem:iso_L} and the classification of Fano threefolds with Picard number two~\cite{SaitoFano}.
\end{proof}

We now formulate and prove a more precise version of Theorem \ref{thm:non_density1}.

\begin{theorem}\label{thm:non_density2}
Let $K$ be a number field and $S$ a finite set of finite places of $K$.
Then there is a proper closed substack $\mathcal Z$ of $\mathcal M$ such that any morphism $\Spec\OO_K[S^{-1}]\to \mathcal M$ factorises over the inclusion $\mathcal Z\to \mathcal M$.
\end{theorem}

\begin{proof}
We may assume that $S$ contains all places above $2$.
Suppose that there is no such  substack, so that there exist a number field $K$ and a finite set of finite places $S$ of $K$ such that $\mathcal M(\OO_K[S^{-1}])$ is Zariski dense. In particular, since the stack $\mathcal L_{\ZZ[1/2]}$ is isomorphic to $\mathcal M_{\ZZ[1/2]}$ (Lemma \ref{lem:isom_of_stacks}), we see that $\mathcal L(\OO_K[S^{-1}])$ is Zariski dense. Next, note that $\mathcal L\to \mathcal C$ is a  smooth proper surjective morphism (Lemma \ref{lem:smoothproper}). In particular, since $\mathcal L(\OO_K[S^{-1}])$ is dense in $\mathcal L$, it follows that $\mathcal C(\OO_K[S^{-1}])$ is dense in $\mathcal C$. As $\mathcal C$ is separated over $\ZZ$, Lemma~\ref{lem:unicity} then implies that the set of Fano threefolds with Picard number $1$, index $2$ and degree $4$ over $K$ with good reduction outside $S$ is infinite. This contradicts our finiteness result for such Fano threefolds (Theorem \ref{thm:finiteness_for_fanos}).
\end{proof}

%\begin{remark}
%The philosophy of our proofs of Theorem \ref{theorem:counter-example} and Theorem \ref{thm:non_density2} can be summarised by the following ``correspondence'' of stacks
%\[\xymatrix{A \ar[r] & \mathcal L \ar[rr]^{\textrm{moduli map}} \ar[d]_{\cong} & & \mathcal C \\ & \mathcal M & &} \] Here the moduli map $\mathcal L\to\mathcal C$ is a smooth proper morphism whose geometric fibres are abelian surfaces, and $A$ is a geometric fibre of $\mathcal L\to \mathcal C$. The morphism $A\to \mathcal C$ is constant (i.e., isotrivial). We claim (in the proof of Theorem \ref{theorem:counter-example}) that $A\to \mathcal L$ is ``non-constant''.
%\end{remark}

\begin{remark}\label{remark:campana} 
We finish with an (imprecise) explanation of how Theorem \ref{thm:non_density2} is in accordance with Campana's generalization of the Lang-Vojta conjecture (see \cite[Conj.~13.23]{Campana}).  Our proof of Theorem \ref{thm:non_density2} shows that the stack $\mathcal M$   admits a smooth fibration over the hyperbolic stack $\mathcal C$, all of whose geometric fibres are abelian surfaces. In particular, this implies that the moduli stack $\mathcal M$ is \emph{non-special} in the sense of Campana \cite[Def.~8.1]{Campana}. 
Ignoring stacky issues, Theorem \ref{thm:non_density2} is therefore in accordance with Campana's generalization of the Lang-Vojta conjecture \cite[Conj.~13.23]{Campana} which asserts that a variety is \emph{non-special} if and only if for all number fields $K$ and finite sets of finite places $S$ of $K$, its set of $\OO_K[S^{-1}]$-points is  not dense. 
\end{remark}

\bibliography{refsci}{}
\bibliographystyle{plain}

\end{document}